\documentclass[12pt]{amsart}
\usepackage{graphicx}
\usepackage{amsmath,amsthm, amsfonts,amssymb, stmaryrd,yfonts,pxfonts,pifont,eufrak, bbm}
\newtheorem{Cor}{Corollary}
 \newtheorem{Lemma}{Lemma}
 
 \newtheorem{ex}{Example}
 
 \theoremstyle{definition}
 
 \theoremstyle{remark}
 \newtheorem{Remark}[Lemma]{Remark}
 \numberwithin{equation}{subsection}
\begin{document}
\title[DYNAMICS OF NON-AUTONOMOUS DISCRETE DYNAMICAL SYSTEMS]{DYNAMICS OF NON-AUTONOMOUS DISCRETE DYNAMICAL SYSTEMS}%
\author{PUNEET SHARMA AND MANISH RAGHAV}
\address{Department of Mathematics, I.I.T. Jodhpur, Old Residency Road, Ratnada, Jodhpur-342011, INDIA}%
\email{puneet.iitd@yahoo.com, manishrghv@gmail.com }%

%\thanks{The first author thanks CSIR for financial support.}%

\subjclass{37B20, 37B55, 54H20}

\keywords{non-autonomous dynamical systems, transitivity, weakly
mixing, topological mixing, topological entropy, Li-Yorke chaos}

\begin{abstract}
In this paper we study the dynamics of a general non-autonomous
dynamical system generated by a family of continuous self maps on a
compact space $X$. We derive necessary and sufficient conditions for
the system to exhibit complex dynamical behavior. In the process we
discuss properties like transitivity, weakly mixing, topologically
mixing, minimality, sensitivity, topological entropy and Li-Yorke
chaoticity for the non-autonomous system. We also give examples to
prove that the dynamical behavior of the non-autonomous system in
general cannot be characterized in terms of the dynamical behavior
of its generating functions.
\end{abstract}
\maketitle

\section{INTRODUCTION}

Let $(X,d)$ be a compact metric space and let $\mathbb{F}= \{f_n: n
\in \mathbb{N} \}$ be a family of continuous self maps on $X$. Any
such family $\mathbb{F}$ generates a non-autonomous dynamical system
via the relation $x_{n}= f_n(x_{n-1})$. Throughout this paper, such
a dynamical system will be denoted by $(X,\mathbb{F})$. For any
$x\in X$, $\{ f_n \circ f_{n-1} \circ \ldots \circ f_1(x) :
n\in\mathbb{N}\}$ defines the orbit of $x$. The objective of study
of a non autonomous dynamical system is to investigate the orbit of
an arbitrary point $x$ in $X$. For notational convenience, let
$\omega_n(x) = f_n\circ f_{n-1}\circ \ldots \circ f_1(x)$ be the
state of the system after $n$ iterations. If $y=\omega_n(x)=
f_n\circ f_{n-1}\circ \ldots \circ f_1(x)$, then, $x\in
f_1^{-1}\circ f_2^{-1}\circ\ldots\circ f_n^{-1}(y)=\omega_n^{-1}(y)$
and hence $\omega_n^{-1}$ traces the
point $n$ units back in time.\\

A point $x$ is called \textit{periodic} for $\mathbb{F}$ if there
exists $n\in\mathbb{N}$ such that $\omega_{nk}(x)=x$ for all $k\in
\mathbb{N}$. The least such $n$ is known as the period of the point
$x$. The system $(X,\mathbb{F})$ is \textit{transitive} (or
$\mathbb{F}$ is transitive) if for each pair of open sets $U,V$ in
$X$, there exists $n \in \mathbb{N}$ such that $\omega_n(U)\bigcap
V\neq \phi$. The system $(X,\mathbb{F})$ is said to be
\textit{minimal} if it does not contain any proper non-trivial
subsystems. The system $(X,\mathbb{F})$ is said to be \textit{weakly
mixing} if for any collection of non-empty open sets $U_1, U_2, V_1,
V_2$, there exists a natural number $n$ such that $\omega_n(U_i)
\bigcap V_i \neq \phi$, $i=1,2$. Equivalently, we say that the
system is weakly mixing if $\mathbb{F}\times\mathbb{F}$ is
transitive. The system is said to be \textit{topologically mixing}
if for every pair of non-empty open sets $U, V$ there exists a
natural number $K$ such that $\omega_n(U) \bigcap V \neq \phi$ for
all $n \geq K$. The system is said to be \textit{sensitive} if there
exists a $\delta>0$ such that for each $x\in X$ and each
neighborhood $U$ of $x$, there exists $n\in \mathbb{N}$ such that
$diam(\omega_n(U))>\delta$. If there exists $K>0$ such that
$diam(\omega_n(U))>\delta$, $~~\forall n\geq K$, then the system is
\textit{cofinitely sensitive}. A set $S$ is said to be
\textit{scrambled} if for any $x,y\in S$,
$\limsup\limits_{n\rightarrow \infty} d(\omega_n(x),\omega_n(y))>0$
but $\liminf\limits_{n\rightarrow \infty}
d(\omega_n(x),\omega_n(y))=0$. A system $(X,\mathbb{F})$ is said to
be \textit{Li-Yorke chaotic} if it contains an uncountable scrambled
set. Incase the $f_n$'s coincide, the above definitions coincide
with the known notions of an autonomous dynamical system. See
\cite{bc,bs,de} for details.\\

We now define the notion of \textit{topological entropy} for a
non-autonomous system $(X,\mathbb{F})$.\\

Let $X$ be a compact space and let $\mathcal{U}$ be an open cover of
$X$. Then $\mathcal{U}$ has a finite subcover. Let $\mathcal{L}$ be
the collection of all finite subcovers and let $\mathcal{U}^*$ be
the subcover with minimum cardinality, say $N_{\mathcal{U}}$. Define
$H(\mathcal{U}) = log N_{\mathcal{U}} $. Then $H(\mathcal{U})$ is
defined as the \textit{entropy} associated with the open cover
$\mathcal{U}$. If $\mathcal{U}$ and $\mathcal{V}$ are two open
covers of $X$, define, $\mathcal{U} \vee \mathcal{V} = \{ U \bigcap
V : U \in \mathcal{U}, V \in \mathcal{V} \}$. An open cover $\beta$
is said to be refinement of open cover $\alpha$ i.e. $\alpha \prec
\beta$, if every open set in $\beta$ is contained in some open set
in $\alpha$. It can be seen that if $\alpha \prec \beta$ then
$H(\alpha) \leq H(\beta)$. For a self map $f$ on $X$, $f^{-1}
(\mathcal{U}) = \{ f^{-1} (U) : U \in \mathcal{U} \}$
is also an open cover of $X$. Define,\\

\centerline{$h _{\mathbb{F}, \mathcal{U}} = \limsup \limits_{k
\rightarrow \infty} \frac{H( \mathcal{U} \vee
\omega_1^{-1}(\mathcal{U}) \vee \omega_2^{-1}(\mathcal{U}) \vee
\ldots \vee \omega_{k-1}^{-1}(\mathcal{U}))}{k}$} \vskip .25cm

Then $\sup h _{\mathbb{F}, \mathcal{U}}$, where $\mathcal{U}$ runs
over all possible open covers of $X$ is known as the
\textit{topological entropy of the system $(X,\mathbb{F})$} and is
denoted by $h(\mathbb{F})$. Incase the maps $f_n$ coincide, the
above definition coincides with the known notion of topological
entropy. See \cite{bc,bs} for details.\\

Let $(X,d)$ be a metric space and let $CL(X)$ denote the collection
of all non-empty closed subsets of $X$. For any two closed subsets
$A,B$ of $X$, define,

\centerline{$ d_H (A, B) = \inf \{ \epsilon >0 : A \subseteq
S_{\epsilon} (B) \text{ and } B \subseteq S_{\epsilon} (A) \} $}

It is easily seen that $d_H$ defined above is a metric on $CL(X)$
and is called \textit{Hausdorff metric}. The metric $d_H$ preserves
the metric on $X$, i.e. $d_H(\{x\}, \{y\}) = d(x,y)$ for all $x,y
\in X$. The topology generated by this metric is known as the
\textit{Hausdorff metric topology} on $CL(X)$ with respect to the
metric $d$ on $X$ \cite{be,mi}. It is known that $\lim
\limits_{n\rightarrow \infty} A_i =A$ if and only if $A_i$ converges
to $A$ under Hausdorff metric\cite{do}.\\

Many of the natural systems occurring in the nature have been
studied using mathematical models. While systems like the logistic
model have been used to characterize the population growth,
continuous systems like the Lorenz model have been used for weather
prediction to a great precision. Although various mathematical
models exploring such systems have been proposed and long term
behavior of such systems has been studied, most of the mathematical
models are autonomous in nature and hence cannot be used to model a
general dynamical system. Thus, there is a strong need to study and
develop the theory of non-autonomous dynamical systems. The theory
of non-autonomous dynamical systems helps characterizing the
behavior of various natural phenomenon which cannot be modeled by
autonomous systems. Some of the studies in this direction have been
made and some results have been obtained. In \cite{sk1} authors
study the topological entropy of a general non-autonomous dynamical
system generated by a family $\mathbb{F}$. In particular authors
study the case when the family $\mathbb{F}$ is equicontinuous or
uniformly convergent. In \cite{sk2} authors discuss minimality
conditions for a non-autonomous system on a compact Hausdorff space
while focussing on the case when the non-autonomous system is
defined on a compact interval of the real line. In \cite{jd} authors
prove that if $f_n\rightarrow f$, in general there is no relation
between chaotic behavior of the non-autonomous system generated by
$f_n$ and the chaotic behavior of $f$. In \cite{bp} authors
investigate properties like weakly mixing, topological mixing,
topological entropy and Li-Yorke chaos for the non-autonomous
system. They prove that the dynamics of a non-autonomous system is
very different from the autonomous case. They also give a few
techniques to study the qualitative behavior of
a non-autonomous system.\\

Although some studies have been made and some useful results have
been obtained, a lot of questioned in the field are still unanswered
and a lot of investigation still needs to be done. In this paper, we
study different possible dynamical notions for a non-autonomous
dynamical system generated by a family $\mathbb{F}$. We prove that
if $\mathbb{F}=\{f_1,f_2,\ldots f_n\}$ is finite, the non-autonomous
system is topological mixing if and only if the autonomous system
$(X,f_n\circ f_{n-1}\circ\ldots\circ f_1)$ is topological mixing. We
also prove that if $(X,f_n\circ f_{n-1}\circ\ldots\circ f_1)$ has
positive topological entropy (is Li-Yorke chaotic) then
$(X,\mathbb{F})$ also has positive topological entropy (is Li-Yorke
chaotic).  We also establish similar results for transitivity/dense
periodicity of the non-autonomous system.  In addition, if
$\mathbb{F}$ is commutative, the non-autonomous system is weakly
mixing if and only if $(X,f_n\circ f_{n-1}\circ\ldots\circ f_1)$ is
weakly mixing. Thus, we prove that if the family $\mathbb{F}$ is
finite, under certain assumptions, the study of non-autonomous
dynamical system can be reduced to the autonomous case. We also
establish alternate criteria to establish weakly mixing/topological
mixing for a general non-autonomous dynamical system. In the end, we
study the dynamical behavior of the system with respect to the
members of the family $\mathbb{F}$. We prove that the dynamical
behavior of the generating members in general does not carry over to
the non-autonomous system generated. While the non-autonomous system
can exhibit a certain dynamical notion without any of the generating
members exhibiting the same, on some instances, the system might not
exhibit certain dynamical behavior even when all the generating
members exhibit the same.

\section{Main Results}

Throughout the paper, let $(X,d)$ be a compact metric space and let
$\mathbb{F}= \{f_n: n \in \mathbb{N}\}$ be a family of surjective
continuous self maps on $X$.\\

We first give some results establishing various dynamical properties
of the non-autonomous system, when the family $\mathbb{F}=
\{f_1,f_2,\ldots,f_n\}$ is finite. It is worth mentioning that when
the family $\mathbb{F}=\{f_1,f_2,\ldots,f_n\}$ is finite, the
non-autonomous dynamical system is generated by the relation $x_k =
f_k(x_{k-1})$ where $f_k = f_{(1+(k-1)\mod n)}$.\\

\begin{Lemma}
$(X,f_n\circ f_{n-1}\circ\ldots\circ f_1)$ has dense set of periodic
points  $\Rightarrow$ $(X,\mathbb{F})$ has dense set of periodic
points.
\end{Lemma}

\begin{proof}
Let $U$ be any  non-empty open subset of $X$. As $f_n\circ
f_{n-1}\circ\ldots\circ f_1$ has dense set of periodic points, there
exists $k\in \mathbb {N}$ and $x\in U$ such that $(f_n\circ
f_{n-1}\circ\ldots\circ f_1)^k(x)=x$. Thus, $\omega_{nk}(x)=x$.
Consequently $\omega_{rnk}(x)=x~~\forall r\geq 1$ and $x$ is also
periodic for $(X,\mathbb{F})$. Hence $(X,\mathbb{F})$ has dense set
of periodic points.
\end{proof}
\vskip 0.25cm

\begin{Lemma}
If $(X,f_n\circ f_{n-1}\circ\ldots\circ f_1)$ is transitive, then
$(X,\mathbb{F})$ is transitive.
\end{Lemma}

\begin{proof}
Let $U,V$ be any pair of non-empty open subsets of $X$, As $f_n\circ
f_{n-1}\circ\ldots\circ f_1$ is transitive, there exists $k\in
\mathbb {N}$ such that $(f_n\circ f_{n-1}\circ\ldots\circ
f_1)^k(U)\cap V \neq \phi$. Consequently $\omega_{nk}(U)\cap V\neq
\phi$ and hence $(X,\mathbb{F})$ is transitive.
\end{proof}
\vskip 0.25cm

The above result establishes the transitivity of the non-autonomous
system, incase the corresponding autonomous system is transitive.
However, the correspondence is one-sided and the converse of the
above result is not true. We give an example in support of our
statement.

\begin{ex}
\label{ex9}
Let $I$ be the unit interval and let $f_1,f_2$ be defined as\\

$f_1(x) = \left\{%
\begin{array}{ll}
            2x+\frac{1}{2}  & \text{for x} \in [0, \frac{1}{4}] \\
           -2x+\frac{3}{2} & \text{for x} \in [\frac{1}{4}, \frac{3}{4}] \\
            2x-\frac{3}{2}  & \text{for x} \in [\frac{3}{4},1] \\
\end{array} \right.$

$f_2(x) = \left\{%
\begin{array}{ll}
        x+\frac{1}{2}             & \text{for x} \in [0, \frac{1}{2}] \\
        -4x+3                     & \text{for x} \in [\frac{1}{2},\frac{3}{4}] \\
       2x-\frac{3}{2} & \text{for x} \in [\frac{3}{4},1] \\
\end{array} \right.$ \\
Let $\mathbb{F}=\{f_1,f_2\}$ and $({X,\mathbb{F}})$ be the
corresponding non-autonomous dynamical system. As $(X,f_2\circ f_1)$
has an invariant set $U=[\frac{1}{2},1]$, $f_2\circ f_1$ is not
transitive. However, as $f_1$ expands every open set $U$ in $[0,1]$
and $f_2$ expands the right half of the unit interval with
$f_2([0,\frac{1}{2}])=[\frac{1}{2},1]$, the non-autonomous system
generated by $\mathbb{F}$ is transitive.

\end{ex}

\begin{Lemma}\label{wm1}
If $\mathbb{F}$ is a commutative family, then,
$\mathbb{F}\times\mathbb{F}$ is transitive if and only if
$\underbrace{\mathbb{F}\times\mathbb{F}\times\ldots\times\mathbb{F}}_{n~~
times}$ is transitive $~~\forall n\geq 2$.
\end{Lemma}

\begin{proof}
Let $\mathbb{F}\times\mathbb{F}$ be transitive. We prove the forward
part with the help of mathematical induction. Let
$\underbrace{\mathbb{F}\times\mathbb{F}\times\ldots\times\mathbb{F}}_{k~~
times}$ be transitive and let $U_1,U_2,\ldots,U_{k+1}$ and
$V_1,V_2,\ldots,V_{k+1}$ be a pair of $k+1$ non-empty open sets in
$X$. As $\mathbb{F}\times\mathbb{F}$ is transitive, there exists
$r>0$ such that $\omega_r(U_k)\cap U_{k+1}\neq \phi$ and
$\omega_r(V_k)\cap V_{k+1}\neq \phi$. Let $U= U_k \cap
\omega_r^{-1}(U_{k+1})$ and $V= V_k \cap \omega_r^{-1}(V_{k+1})$.
Then $U$ and $V$ are non-empty open sets in $X$. Also as
$\underbrace{\mathbb{F}\times\mathbb{F}\times\ldots\times\mathbb{F}}_{k~~
times}$ is transitive, there exists $t>0$ such that
$\omega_t(U_i)\cap V_i\neq \phi$ for $i=1,2,\ldots,k-1$ and
$\omega_t(U)\cap V \neq \phi$.

As $U\subset U_k$ and $V\subset V_k$, we have $\omega_t(U_k)\cap
V_k\neq \phi$. Also $\omega_t(U)\cap V \neq \phi$ implies
$\omega_r(\omega_t(U))\cap \omega_r(V)\neq \phi$. As $f_i$ commute
with each other,  we have $\omega_t(\omega_r(U))\cap \omega_r(V)\neq
\phi$. As $\omega_r(U)\subseteq U_{k+1}$ and $\omega_r(V)\subset
V_{k+1}$, we have $\omega_t(U_{k+1})\cap V_{k+1}\neq \phi$.
Consequently $\omega_t(U_i)\cap V_i\neq \phi$ for $i=1,2,\ldots,k+1$
and hence
$\underbrace{\mathbb{F}\times\mathbb{F}\times\ldots\times\mathbb{F}}_{k+1~~
times}$ is transitive.\\

Proof of converse is trivial as if
$\underbrace{\mathbb{F}\times\mathbb{F}\times\ldots\times\mathbb{F}}_{n~~
times}$ is transitive $~~\forall n\geq 2$, in particular taking
$n=2$ yields $\mathbb{F}\times \mathbb{F}$ is transitive.

\end{proof}

\begin{Remark}
For autonomous systems, it is known that $f\times f$ is transitive,
then $\underbrace{f\times f\times\ldots\times f}_{n times}$ is
transitive for all $n\geq 2$\cite{ba} and hence the result
established above is an analogous extension of the autonomous case.
It may be noted that the proof uses the commutative property of the
members of the family $\mathbb{F}$ and hence is not true for a
non-autonomous system generated by any general family $\mathbb{F}$.
However, the proof does not use the finiteness of the family
$\mathbb{F}$ and hence the result holds even when the generating
family  $\mathbb{F}$
is infinite.\\
\end{Remark}

\begin{Lemma}\label{wm2}
If $\mathbb{F}$ is a commutative family, then $(X,\mathbb{F})$ is
weakly mixing if and only if for any finite collection of non-empty
open sets $\{U_1,U_2,\ldots,U_m\}$, there exists a subsequence
$(r_n)$ of positive integers such that $\lim \limits_{n \rightarrow
\infty} \omega_{r_n} (U_i) =X, ~~~\forall i=1,2,\ldots,m$.
\end{Lemma}

\begin{proof}
Let $n\in \mathbb{N}$ be arbitrary and let $\{U_1,U_2,\ldots,U_m\}$
be any finite collection of non-empty open sets of $X$. As $X$ is
compact, there exist $x_1,x_2,\ldots x_{k_n}$ such that
$X=\bigcup\limits_{i=1}^{k_n} S(x_i, \frac{1}{n})$. As
$(X,\mathbb{F})$ is weakly mixing, by lemma \ref{wm1}, there exists
$r_n>0$ such that $\omega_{r_n}(U_i)\cap S(x_j,\frac{1}{n})\neq
\phi~~~ \forall i,j$ and hence for any $i$,
$d_H(\omega_{r_n}(U_i),X)\leq \frac{1}{n}$. As $n\in \mathbb{N}$ is
arbitrary, $\lim \limits_{n \rightarrow \infty} \omega_{r_n}
(U_i)=X~~\forall i~~$ and the proof
for the forward part is complete.\\

Conversely, let $U_1,U_2$ and $V_1,V_2$ be a pair of $2$ non-empty
open subsets of $X$. For $~~~i=1,2$, let $v_i\in V_i$ and let
$\epsilon>0$ such that $S(v_i,\epsilon)\subset V_i$. By given
condition, there exists a subsequence $(r_n)$ of natural numbers
such that $\lim \limits_{n \rightarrow \infty} \omega_{r_n} (U_i)
=X$ for $i=1,2$. Thus, there exists $r_k$ such that
$d_H(\omega_{r_k}(U_i),X)< \frac{\epsilon}{2},~~~i=1,2$.
Consequently $ \omega_{r_k}(U_i)\cap V_i\neq \phi$ and hence
$(X,\mathbb{F})$ is weakly mixing.
\end{proof}

\begin{Remark}
It may be noted that the proof of converse does not need
commutativity of the family $\mathbb{F}$. However, to establish the
forward part, we use lemma \ref{wm1} and hence use the commutativity
of the family $\mathbb{F}$. Thus, the result may not hold good when
considered for a general non-autonomous system. Also, the result
does not use finiteness condition on $\mathbb{F}$ and hence is valid
even when the system is generated by an infinite family
$\mathbb{F}$.
\end{Remark}

\begin{Remark}
It is known that an autonomous system is weakly mixing if and only
if for any non-empty open set $U$, there exists a subsequence
$(r_n)$ of positive integers such that $\lim \limits_{n \rightarrow
\infty} \omega_{r_n} (U) =X$ \cite{do}.  Thus for non-autonomous
case, the result above establishes an stronger extension of the
result proved in the autonomous case. However, the above result also
holds when the maps $f_n$ coincide and hence a stronger version of
the result in \cite{do} is true for the autonomous case. For the
sake of completeness, we mention the obtained result below.
\end{Remark}

\begin{Cor}
A continuous self map $f$ is weakly mixing if and only if for any
finite collection of non-empty open sets $\{U_1,U_2,\ldots,U_m\}$,
there exists a subsequence $(r_n)$ of positive integers such that
$\lim \limits_{n \rightarrow \infty} \omega_{r_n} (U_i) =X,
~~~\forall i=1,2,\ldots,m$.
\end{Cor}

\begin{Lemma}
$(X,\mathbb{F})$ is topologically mixing if and only if for each
non-empty open set $U$, $\lim \limits_{n \rightarrow \infty}
\omega_{n}(U) =X$.
\end{Lemma}

\begin{proof}
Let $n\in \mathbb{N}$ be arbitrary and let $U$ be any non-empty open
subset of $X$. As $X$ is compact, there exist $x_1,x_2,\ldots
x_{k_n}$ such that $X=\bigcup\limits_{i=1}^{k_n} S(x_i,
\frac{1}{n})$. As $\mathbb{F}$ is topologically mixing, there exists
$M_i,~~i=1,2,\ldots,k_n$ such that $\omega_k(U)\cap
S(x_i,\frac{1}{n}) \neq \phi~~~ \forall k\geq M_i$. Let $M=\max\{M_i
: 1\leq i\leq k_n\}$. Then $\omega_k(U)\cap S(x_i,\frac{1}{n}) \neq
\phi~~~ \forall k\geq M$. Consequently
$d_H(\omega_k(U),X)<\frac{1}{n} ~~~ \forall k\geq M$. As $n\in
\mathbb{N}$ is arbitrary, $\lim \limits_{n \rightarrow \infty}
\omega_{n}(U) =X$ and the proof of forward part is complete.

Conversely, let $U,V$ be a any pair of non-empty open subsets of
$X$. Let $v\in V$ and let $\epsilon>0$ be such that
$S(v,\epsilon)\subset V$. By given condition, $\lim \limits_{n
\rightarrow \infty} \omega_{n}(U) =X$. Thus, there exists $K>0$ such
that $d_H(\omega_{k}(U),X)< \frac{\epsilon}{2}~~~\forall k\geq K$.
Consequently $ \omega_{k}(U)\cap V\neq \phi~~~\forall n\geq K$ and
hence $(X,\mathbb{F})$ is topologically mixing.
\end{proof}
\vskip 0.25cm

\begin{Remark}
In \cite{do}, the authors establish that an autonomous system
$(X,f)$ is topologically mixing if and only if for each non-empty
open set $U$, $\lim \limits_{n \rightarrow \infty} f^{n}(U) =X$.
Once again, we prove that an analogous result does hold when
considered for a general non-autonomous system. However, it may be
noted that commutativity or finiteness of the family $\mathbb{F}$
were not needed to establish the above result and hence the result
holds for a general non-autonomous dynamical system.
\end{Remark}
\vskip 0.25cm

\begin{Lemma}
\label{lem18} If $\mathbb{F}=\{f_1,f_2,\ldots,f_n\}$ is a finite
commutative family, then, $(X,\mathbb{F})$ is weakly mixing if and
only if $(X,f_n\circ f_{n-1}\circ\ldots\circ f_1)$ is weakly mixing.
\end{Lemma}

\begin{proof}
Let $U$ be a non-empty open subset of $X$. We will equivalently
prove that there exists a sequence $(z_k)$ of natural numbers such
that $\lim \limits_{k\rightarrow \infty} (f_n\circ
f_{n-1}\circ\ldots\circ f_1)^{z_k}(U)=X$. As $(X,\mathbb{F})$ is
weakly mixing, by lemma $5$, there exists sequence $(s_k)$ such that
$\lim \limits_{k\rightarrow \infty} \omega_{s_k}(U)=X$. Also the
family $\mathbb{F}$ is finite and hence there exists $l\in
\{1,2,\ldots,n\}$ and a subsequence $(m_k)$ of $(s_k)$, $m_k =
l+r_kn~~$ such that $\lim \limits_{k\rightarrow \infty} f_l\circ
f_{l-1}\circ\ldots\circ f_1\circ \omega_{r_k n}(U)=X$. As each $f_i$
are surjective, $\lim \limits_{k\rightarrow \infty} \omega_{(r_k+1)
n}(U)=X$. Consequently $\lim \limits_{k\rightarrow \infty} (f_n\circ
f_{n-1}\circ\ldots\circ f_1)^{r_k+1}(U)=X$
and $(X,f_n\circ f_{n-1}\circ\ldots\circ f_1)$ is weakly mixing.\\

Conversely, let $U_1,U_2,V_1,V_2$ be any two pairs of non-empty open
subsets of $X$, As $f_n\circ f_{n-1}\circ\ldots\circ f_1$ is weakly
mixing, there exists $k\in \mathbb {N}$ such that $(f_n\circ
f_{n-1}\circ\ldots\circ f_1)^k(U_i)\cap V_i \neq \phi$ for $i=1,2$.
Consequently $\omega_{nk}(U_i)\cap V_i\neq \phi$ for $i=1,2$ and
hence $(X,\mathbb{F})$ is weakly mixing.
\end{proof}
\vskip 0.25cm
\begin{Remark}
The result establishes the equivalence of the weakly mixing of the
non-autonomous system $(X,\mathbb{F})$ and the autonomous system
$(X,f_n\circ f_{n-1}\circ\ldots\circ f_1)$. It may be noted that as
the proof uses the lemma $5$ proved earlier, commutativity of the
family $\mathbb{F}$ cannot be relaxed. Thus the result may not hold
good if the assumptions in the hypothesis are relaxed.
\end{Remark}
\vskip 0.25cm

\begin{Remark}
It may be noted that the above result uses the surjectivity of the
maps $f_i$. Thus, if the maps are not surjective, the above result
may not hold, i.e. the non-autonomous system may exhibit weakly
mixing even if the system $(X,f_n\circ f_{n-1}\circ\ldots\circ f_1)$
is not weakly mixing. We now give an example in support of our
statement.
\end{Remark}

\begin{ex}
Let $I$ be the unit interval and let $f_1,f_2$ be defined as\\

$f_1(x) = \left\{%
\begin{array}{ll}
            2x  & \text{for x} \in [0, \frac{1}{2}] \\
           -x+\frac{3}{2} & \text{for x} \in [\frac{1}{2},1] \\
\end{array} \right.$

$f_2(x) = \left\{%
\begin{array}{ll}
        -2x+\frac{1}{2}             & \text{for x} \in [0, \frac{1}{4}] \\
         2x-\frac{1}{2}             & \text{for x} \in [\frac{1}{4},\frac{1}{2}] \\
         -2x+\frac{3}{2}             & \text{for x} \in [\frac{1}{2},\frac{3}{4}] \\
         2x-\frac{3}{2}             & \text{for x} \in [\frac{3}{4},1] \\
\end{array} \right.$ \\

Let $\mathbb{F}$ be a finite family of maps $f_1$ and $f_2$ defined
above. As $[0,\frac{1}{2}]$ is invariant for $f_2\circ f_1$, the map
$f_2\circ f_1$ does not exhibit any of the mixing properties.
However, for any open set $U$ in $[0,1]$, there exists $k\in
\mathbb{N}$ such that $(f_2\circ f_1)^k (U)=[0,\frac{1}{2}]$.
Consequently, $\omega_{2k+1}(U)=[0,1]$. As the argument holds for
any  odd integer greater than $k$, the non-autonomous system is
weakly mixing.

\end{ex}

\begin{Lemma}
\label{lem20} If $\mathbb{F}=\{f_1,f_2,\ldots,f_n\}$ is a finite
family, then, $(X,\mathbb{F})$ is topologically mixing if and only
if $(X,f_n\circ f_{n-1}\circ\ldots\circ f_1)$ is topologically
mixing.
\end{Lemma}

\begin{proof}
Let $U$ be a non-empty open subset of $X$. We will equivalently
prove that $\lim \limits_{k\rightarrow \infty} (f_n\circ
f_{n-1}\circ\ldots\circ f_1)^k(U)=X$. As $(X,\mathbb{F})$ is
topologically mixing, by lemma $8$, $\lim \limits_{k\rightarrow
\infty} \omega_{k}(U)=X$. In particular $\lim \limits_{k\rightarrow
\infty} \omega_{nk}(U)=X$ or $\lim \limits_{k\rightarrow \infty}
(f_n\circ f_{n-1}\circ\ldots\circ f_1)^k(U)=X$ and hence
$(X,f_n\circ
f_{n-1}\circ\ldots\circ f_1)$ is topologically mixing.\\

Conversely, $U$ be a non-empty open subset of $X$. We will
equivalently prove that $\lim \limits_{k\rightarrow \infty}
\omega_k(U)=X$. As $f_n\circ f_{n-1}\circ\ldots\circ f_1$ is
topologically mixing, $\lim \limits_{k\rightarrow \infty} (f_n\circ
f_{n-1}\circ\ldots\circ f_1)^k(U)=X$. Consequently, $\lim
\limits_{k\rightarrow \infty} \omega_{nk}(U)=X$. As each $f_i$ are
surjective, by continuity we have for each $l\in \{1,2,\ldots,n\}$,
$f_l\circ f_{l-1}\circ\ldots \circ f_1 (\lim \limits_{k\rightarrow
\infty} \omega_{nk}(U))= \lim \limits_{k\rightarrow \infty}
(f_l\circ f_{l-1}\circ\ldots \circ f_1\circ \omega_{nk}(U))= X$.
Consequently $\lim \limits_{k\rightarrow \infty} \omega_k(U)=X$ and
$(X,\mathbb{F})$ is topologically mixing.
\end{proof}

\begin{Remark}
The result once again is an analogous extension of the autonomous
case. The result proves that the identical conclusion can be made
for the non-autonomous case without strengthening the hypothesis. It
is worth noting that the result does not use commutativity of
$\mathbb{F}$ and hence asserts the complex nature of a topological
mixing in a general dynamical system.
\end{Remark}

\noindent

In \cite{sk1}, authors prove that for
$\mathbb{F}=\{f_1,f_2,\ldots,f_n\}$ is a finite family, then,
$h(\mathbb{F})= \frac{1}{n} h(f_n\circ f_{n-1}\circ\ldots\circ
f_1)$. However, as the authors of this paper were not aware of the
result while addressing the problem, for the sake of completion, we
include the proof here.

\begin{Lemma}
If $\mathbb{F}=\{f_1,f_2,\ldots,f_n\}$ is a finite family, then,
$h(\mathbb{F})\geq \frac{1}{n} h(f_n\circ f_{n-1}\circ\ldots\circ
f_1)$. Consequently if the associated autonomous system has positive
topological entropy, the non-autonomous system also has a positive
topological entropy.
\end{Lemma}

\begin{proof}
For any open cover $\mathcal{U}$ of $X$, the entropy of the system
with respect to the open cover $\mathcal{U}$ is defined as\\

\noindent $h _{\mathbb{F}, \mathcal{U}} = \liminf \limits_{k
\rightarrow \infty} \frac{H( \mathcal{U} \vee
\omega_1^{-1}(\mathcal{U}) \vee \omega_2^{-1}(\mathcal{U}) \vee
\ldots \vee \omega_{k-1}^{-1}(\mathcal{U}))}{k}= \liminf \limits_{k
\rightarrow \infty} \frac{H( \mathcal{U} \vee
\omega_1^{-1}(\mathcal{U}) \vee \omega_2^{-1}(\mathcal{U}) \vee
\ldots \vee
\omega_{nk-1}^{-1}(\mathcal{U}))}{nk}$ \\

\noindent Also as $\mathcal{U} \vee \omega_{n}^{-1}(\mathcal{U})
\vee \omega_{2n}^{-1}(\mathcal{U}) \vee \ldots \vee
\omega_{n(k-1)}^{-1}(\mathcal{U}) \prec \mathcal{U} \vee
\omega_1^{-1}(\mathcal{U}) \vee \omega_2^{-1}(\mathcal{U}) \vee
\ldots \vee \omega_{nk-1}^{-1}(\mathcal{U})$, we have \\

\noindent $H(\mathcal{U} \vee \omega_{n}^{-1}(\mathcal{U}) \vee
\omega_{2n}^{-1}(\mathcal{U}) \vee \ldots \vee
\omega_{n(k-1)}^{-1}(\mathcal{U})) \leq H( \mathcal{U} \vee
\omega_1^{-1}(\mathcal{U}) \vee \omega_2^{-1}(\mathcal{U}) \vee
\ldots \vee \omega_{nk-1}^{-1}(\mathcal{U}))$\\

Therefore, \\ $\liminf \limits_{k\rightarrow \infty} \frac{
H(\mathcal{U} \vee \omega_{n}^{-1}(\mathcal{U}) \vee
\omega_{2n}^{-1}(\mathcal{U}) \vee \ldots \vee
\omega_{n(k-1)}^{-1}(\mathcal{U}))}{nk} \leq \liminf
\limits_{k\rightarrow \infty} \frac{H( \mathcal{U} \vee
\omega_1^{-1}(\mathcal{U}) \vee \omega_2^{-1}(\mathcal{U}) \vee
\ldots \vee \omega_{nk-1}^{-1}(\mathcal{U}))}{nk}$\\ Consequently,\\
$\frac{1}{n} \liminf \limits_{k\rightarrow \infty} \frac{H(
\mathcal{U} \vee (f_n\circ f_{n-1}\circ\ldots\circ
f_1)^{-1}(\mathcal{U}) \vee (f_n\circ f_{n-1}\circ\ldots\circ
f_1)^{-2}(\mathcal{U}) \vee \ldots \vee (f_n\circ
f_{n-1}\circ\ldots\circ f_1)^{(-k+1)}(\mathcal{U}))}{k}\\
 \leq \liminf
\limits_{k\rightarrow \infty} \frac{H( \mathcal{U} \vee
\omega_1^{-1}(\mathcal{U}) \vee \omega_2^{-1}(\mathcal{U}) \vee
\ldots \vee \omega_{nk-1}^{-1}(\mathcal{U}))}{nk}$\\ or $\frac{1}{n}
H(f_n\circ f_{n-1}\circ\ldots\circ f_1, \mathcal{U})\leq
H(\mathbb{F},\mathcal{U})$. As $\mathcal{U}$ was arbitrary,
$h(\mathbb{F})\geq \frac{1}{n} h(f_n\circ f_{n-1}\circ\ldots\circ
f_1)$ and the proof is complete.
\end{proof}

\begin{Lemma}
$(X,f_n\circ f_{n-1}\circ\ldots\circ f_1)$ is Li-Yorke chaotic
$\Rightarrow$ $(X,\mathbb{F})$ is Li-Yorke chaotic.
\end{Lemma}

\begin{proof}
Let  Let $(X,f_n\circ f_{n-1}\circ\ldots\circ f_1)$ be Li-Yorke
chaotic and let $S$ be an uncountable scrambled set for $g= f_n\circ
f_{n-1}\circ\ldots\circ f_1$. Consequently for any $x,y\in S$ there
exists a sequence $(r_k)$ and $(s_k)$ of natural numbers such that
$\lim \limits_{k\rightarrow \infty} d(g^{r_k}(x),g^{r_k}(y))> 0 $
and $\lim \limits_{k\rightarrow \infty}d(g^{s_k}(x),g^{s_k}(y))=0$.
Consequently $\lim \limits_{k\rightarrow \infty} d(\omega_{r_k
n}(x),\omega_{r_k n}(y))> 0 $ and $\lim \limits_{k\rightarrow
\infty}d(\omega_{s_k n}(x),\omega_{s_k n}(y))=0$
 and hence $(X,\mathbb{F})$ is Li-Yorke chaotic.
\end{proof}

In general, studying/characterizing the dynamical behavior of a
non-autonomous system is difficult. However, if the generating
functions $f_i$ are surjective, the above results show that under
certain conditions, some of the dynamical properties of the
non-autonomous system can be studied using its generating functions.
Further, if the generating functions are finite, under certain
conditions, some of the dynamical properties of the non-autonomous
systems can be studied (in many cases characterized) using
autonomous systems.

We now study dynamics of the non-autonomous system in terms of its
components $f_i$. We prove that even if the individual maps $f_k$
exhibit certain dynamical behavior, the system $(X,\mathbb{F})$ may
not exhibit similar dynamical behavior.

\begin{ex}
\label{ex1}
Let $\sum =\{0,1\}^{\mathbb{N}}$ be the collection of two-sided
sequences of $0$ and $1$ endowed with the product topology. Let
$\sigma_1,\sigma_2 : \sum \rightarrow \sum$ be defined as
$\sigma_1(\ldots x_{-2}x_{-1}.x_0 x_1 x_2\ldots)= (\ldots
x_{-2}x_{-1}x_0. x_1 x_2\ldots)$ and $\sigma_2(\ldots
x_{-2}x_{-1}.x_0 x_1 x_2\ldots)= (\ldots x_{-2}.x_{-1}x_0 x_1
x_2\ldots)$. Then $\sigma_1, \sigma_2$ are the shift operators and
are continuous with respect to the product topology. Let
$\mathbb{F}=\{\sigma_1, \sigma_2\}$ and let $(X,\mathbb{F})$  be the
corresponding non-autonomous system. It can be seen that each
$\sigma_i$ is transitive. However as $\sigma_1\circ \sigma_2=id$,
the system generated is not transitive.
\end{ex}

\begin{Remark}
The above example proves that a non-autonomous dynamical system may
not be transitive even if each of its generating systems exhibits
the same. It can also be seen that each of the functions are
Li-Yorke chaotic. However, as $\sigma_2\circ \sigma_1=id$, the
system $(X,\mathbb{F})$ fails to be Li-Yorke chaotic. Thus, the
example also shows that the system generated may not exhibit
Li-Yorke chaoticity even if each of the generating functions are
Li-Yorke chaotic.
\end{Remark}

\begin{ex}
\label{ex2}
Let $I$ be the unit interval and let $f_1,f_2: I \rightarrow I$ be
defined as\\

$f_1(x) = \left\{%
\begin{array}{ll}
   2x &\mbox{ if $x\in [0,\frac{1}{2}]$ } \\
  \frac{3}{2}-x &\mbox{ if $x \in [\frac{1}{2}, 1]$ }
\end{array} \right.$

$ f_2(x) = \left\{%
\begin{array}{ll}
   \frac{1}{2}-x &\mbox{ if $x\in [0,\frac{1}{2}]$ } \\
  2x-1 &\mbox{ if $x \in [\frac{1}{2}, 1]$ }
\end{array} \right.$\\

Let $\mathbb{F} = \{f_1,f_2\}$ and let $(X,\mathbb{F})$ be the
corresponding non-autonomous system. As $[\frac{1}{2},1]$ and
$[0,\frac{1}{2}]$ are invariant for $f_1$ and $f_2$ respectively,
none of the $f_i$ are transitive. However, the map

$f_2\circ f_1(x) = \left\{%
\begin{array}{ll}
   \frac{1}{2}-2x &\mbox{ if $x\in [0,\frac{1}{4}]$ } \\
  4x-1 &\mbox{ if $x \in [\frac{1}{4}, \frac{1}{2}]$ }\\
  2-2x &\mbox{ if $x \in [\frac{1}{2}, 1]$ }
\end{array} \right.$

is transitive and hence the non-autonomous system $(X,\mathbb{F})$
is transitive.
\end{ex}

\begin{Remark}
The  example ~\ref{ex1} shows that even if each of the maps $f_i$ are
transitive, the non-autonomous system generated by $\mathbb{F} =
\{f_n : n\in \mathbb{N}\}$ may not be transitive. On the other hand
example ~\ref{ex2}  shows that the non-autonomous system can exhibit
transitivity without any of the maps $f_i$ being transitive. Thus,
transitivity in general cannot be characterized in terms of
transitivity of its generating components $f_i$.
\end{Remark}

\begin{ex}
\label{ex3}
Let $S^1$ be the unit circle and let $\theta\in (0,1)$ be an
rational. Let $f_n: S^1 \rightarrow S^1$ be defined as
$f_n(\phi)=\phi+2\pi \frac{\theta}{3^n}$. As $\theta$ is rational,
each map $f_k$ has dense set of periodic points. However, as $\sum
\limits_{n=1}^{\infty} \frac{\theta}{3^n} < 1$, for any $\beta \in
S^1$,~~~ $f^n(\beta)\neq \beta ~~ \forall n$. Hence the
non-autonomous system generated by $\mathbb{F} = \{f_n : n\in
\mathbb{N}\}$ fails to have any periodic point.
\end{ex}

\begin{ex}
\label{ex4}
Let $S^1$ be the unit circle and let $\theta\in (0,1)$ be an
irrational. Let $f_1,f_2: S^1 \rightarrow S^1$ be defined as
$f_1(\phi)=\phi+2\pi \theta$ and $f_2(\phi)=\phi-2\pi \theta$
respectively and let  $(X,\mathbb{F})$ be the corresponding
non-autonomous dynamical system. As each $f_i$ is an irrational
rotation, no point is periodic for any $f_i$. However as $f_1\circ
f_2 = Id$, the system $( S^1, \mathbb{F})$ has dense set of periodic
points.
\end{ex}

\begin{Remark}
The above examples ~\ref{ex3}  and ~\ref{ex4}  prove that dense
periodicity for a non-autonomous dynamical system cannot be
characterized in terms of dense periodicity of the generating
functions. While example ~\ref{ex4} shows system may exhibit dense
periodicity without any of the generating functions exhibiting the
same, example ~\ref{ex3} proves that the system may fail to have a
dense set of periodic points even when all its generating functions
have the same. Also, it may be noted that as $\theta$ is irrational,
$f_1$ and $f_2$ are also minimal. However, as $f_2\circ f_1= id$,
the system $(X,\mathbb{F})$ fails to be minimal. Thus, the example
also shows that the system generated by a set of minimal systems may
not be minimal.
\end{Remark}

\begin{ex}
\label{ex5}
Let $I$ be the unit interval and let $(q_n)$ be an enumeration of
rationals in I. Let $f_n: I\rightarrow I$ be defined as
$f_n(x)=q_n$~~ for all $x\in I$. Then each $f_n$ is a constant map
but the system $(X,\mathbb{F})$ generated by $\mathbb{F}= \{f_n :
n\in \mathbb{N}\}$ is minimal.
\end{ex}

\begin{Remark}
Once again, example ~\ref{ex4} shows that even if each of the maps $f_i$
are minimal, the non-autonomous system generated by $\mathbb{F}$
need not be minimal. On the other hand, example ~\ref{ex5}  shows that the
non-autonomous system can exhibit minimality without any of the maps
$f_i$ being minimal. Thus, minimality in general cannot be
characterized in terms of minimality of its generating functions.
\end{Remark}

\begin{ex}
\label{ex6}
Let $I$ be the unit interval and let $f_1,f_2$ be defined as\\

$f_1(x) = \left\{%
\begin{array}{ll}
            2x+\frac{1}{2}  & \text{for x} \in [0, \frac{1}{4}] \\
           -2x+\frac{3}{2} & \text{for x} \in [\frac{1}{4}, \frac{3}{4}] \\
            2x-\frac{3}{2}  & \text{for x} \in [\frac{3}{4},1] \\
\end{array} \right.$

$f_2(x) = \left\{%
\begin{array}{ll}
        2x             & \text{for x} \in [0, \frac{1}{2}] \\
       -x+\frac{3}{2} & \text{for x} \in [\frac{1}{2},1] \\
\end{array} \right.$

Let $\mathbb{F} = \{f_1,f_2\}$ and let $(X,\mathbb{F})$ be the
corresponding non-autonomous system. It can be seen that none of the
maps $f_i$ are weakly mixing. However, for any open set $U$, there
exists a natural number $n$ such that $\omega_n(U)=[0,1]$. Hence the
non-autonomous system $(X,\mathbb{F})$ is weakly mixing.
\end{ex}

\begin{Remark}
The non-autonomous dynamical system generated above also exhibits
topological mixing. Thus the example also proves that the
non-autonomous system generated can be weakly mixing (topologically
mixing) without any of its components $f_i$ exhibiting the same.
Also example ~\ref{ex1}  shows that the non-autonomous system generated need
not exhibit weakly mixing (topological mixing) even if each of the
generating functions exhibit weakly mixing (topological mixing).
This proves that in general weakly mixing (topologically mixing) of
a non-autonomous system cannot be characterized in terms of weakly
mixing/ topologically mixing of its components.
\end{Remark}

\begin{ex}
\label{ex7}
Let $I \times S^1$ be the unit cylinder. Let $f_1,f_2: I\times S^1
\rightarrow I \times S^1$ be defined as
$f_1((r,\theta))=(r,\theta+r)$ and $f_2((r,\theta))=(r,\theta-r)$
respectively. Let $\mathbb{F} = \{f_1,f_2\}$ and let
$(X,\mathbb{F})$ be the corresponding non-autonomous system. As
points at different heights of the cylinder are rotating with
different speeds, each of the maps $f_i$ are cofinitely sensitive
\cite{pa}. However as $f_2\circ f_1 =Id$, the system $(I \times S^1,
\mathbb{F})$ is not sensitive.
\end{ex}

\begin{Remark}
Example ~\ref{ex7} shows that even if each of the maps $f_i$ are sensitive,
the non-autonomous system generated need not be sensitive. Also,
example ~\ref{ex2}  proves that the non-autonomous system can exhibit
sensitivity without any of the maps $f_i$ being sensitive. Thus
sensitivity of the non-autonomous system also in general cannot be
characterized in terms of sensitivity of its generating functions.
\end{Remark}

\begin{ex}
\label{ex8}
Let $f_1,f_2: \mathbb{R}\rightarrow \mathbb{R}$  be defined as
$f_1(x)=|x|$ and $f_2(x)=2x-1$. Let $\mathbb{F} = \{f_1,f_2\}$ and
let $(X,\mathbb{F})$ be the corresponding non-autonomous system.
Then $f_1$ and $f_2$ fail to be Li-Yorke chaotic. However, as
$f_2(f_1(-\frac{7}{9}))=\frac{5}{9},
f_2(f_1(\frac{5}{9}))=\frac{1}{9},
f_2(f_1(\frac{1}{9}))=-\frac{7}{9}$, the map $f_2\circ f_1(x):R
\rightarrow R$ poseeses a period $3$ point and hence is Li-Yorke
Chaotic. Consequently, $(X,\mathbb{F})$ is Li-Yorke chaotic.
\end{ex}

\begin{Remark}
The above example shows that the non-autonomous system may be
Li-Yorke chaotic without the generating members being Li-Yorke
chaotic. Also, example ~\ref{ex1}  shows that the non-autonomous system may
not be Li-Yorke chaotic even when all the generating functions are
Li-Yorke chaotic. Thus, Li-Yorke chaoticity of a non-autonomous
system cannot be characterized in terms of Li-Yorke chaoticity of
its generating functions.
\end{Remark}

\section{Conclusion}

In this paper, dynamics of the non-autonomous system generated by a
family $\mathbb{F}$ of continuous self maps on a compact metric
space is discussed. Properties like dense periodicity, transitivity,
weakly mixing, topologically mixing, Li-Yorke chaoticity and
topological entropy are studied and investigated. For a commutative
finite family, we proved that some of the stronger notions of mixing
for the non-autonomous system can be studied using autonomous
systems. We also established that characterization of properties
like weakly mixing also holds analogously in the non-autonomous
case, if the generating family is commutative. Similar
characterization is proved for topological mixing for a general
non-autonomous dynamical system asserting the complex behavior of a
non-autonomous topologically mixing system. It is also observed that
the dynamics of the non-autonomous system generated by the family
$\mathbb{F}$ cannot be characterized in terms of the dynamics of the
generating functions. While the non-autonomous system can exhibit a
certain dynamical behavior without any of the generating functions
exhibiting the same, non-autonomous system may fail to exhibit a
dynamical behavior even if all the generating functions exhibit the
same.

\bibliography{xbib}

\end{document}